\documentclass[12pt]{article}
\usepackage{amsmath, amssymb, amsthm, enumerate} 

\newtheorem{theorem}{Theorem}[section]

\newtheorem{lemma}[theorem]{Lemma}

\newtheorem{problem}[theorem]{Problem}

\def\rr{{\mathbb R}}

\def\nn{{\cal N}}
\def\ii{{\cal I}}

\def\su{\subset}
\def\sp{\supset}
\def\se{\setminus}

\def\de{\delta}
\def\ep{\varepsilon}
\def\si{\sigma}

\def\cd{\cdot}
\def\stb{,\ldots ,}
\def\emp{\emptyset}

\def\ol{\overline}

\def\cl{{\rm cl}\, }
\def\inter{{\rm int}\, }

\def\proof{{\bf Proof.} }

\begin{document}
\title{Translating measurable sets}
\author{M. Laczkovich}


\maketitle

\begin{abstract}
We prove that if $A,B$ are compact subsets of $\rr$ such that the upper density
of $B$ is positive at every point of $B$, then there is a closed null set
$N\su A$ such that $N+B=A+B$. As a corollary we find that if $A,B\su \rr$ are
measurable, and every null subset $N$ of $A$ can be translated into $B$ (that
is, if $B$ contains a suitable translate of $N$), then there is a null set $N_0$
such that $A\se N_0$ can be translated into $B$. The topic is related to some
consistency results of the theory of additive properties of the reals.
\end{abstract}

\section{Introduction and main results}
We say that the set {\it $A\su \rr$ can be translated into the set $B$} if
there exists a real number $x$ such that $A+x=\{a+x\colon a\in A\} \su B$.
The Lebesgue measure of the set $A\su \rr$ is denoted by $|A|$. A set
$N\su \rr$ is {\it null}, if $|N|=0$.

We start with the following question. Suppose $A$ and $B$ are measurable subsets
of $\rr$ such that every null subset of $A$ can be translated into $B$.
Does this condition imply that $A$ can be translated into $B$? The answer is
negative: if $N$ is a singleton (or more generally a nonempty and countable set)
and $B=\rr \se N$, then every null set can be translated into $B$, but $\rr$
cannot be translated into $B$. We prove, however, the following.
\begin{theorem}\label{t1}
If $A,B\su \rr$ are measurable, and every null subset of $A$ can be translated
into $B$, then there is a null set $N\su A$ such that $A\se N$ can be translated
into $B$.
\end{theorem}
If $B$ is closed, then the statement of Theorem \ref{t1} is obvious; moreover,
we can conclude that $A$ can be translated into $B$. Indeed, let $C$ be a
countable subset of $A$ such that the closure of $C$ equals $A$. Clearly, if
$C$ can be translated into $B$ and $C+x\su B$, then $A+x\su B$. We show that
the same conclusion holds if $B$ is a $G_\de$ set which is closed in the density
topology.

We denote by $d(A,x)$ and by $\ol d (A,x)$ the density and the upper density
of the set $A\su \rr$ at the point $x$. A set $G\su \rr$ is said to be {\it
$d$-open} if $G$ is Lebesgue measurable and $d(A,x)=1$ for every $x\in A$.
The {\it density topology} is defined as the collection of all $d$-open sets.
It is well known that the density topology is a completely regular Hausdorff
topology on $\rr$ (see \cite[pp. 148-152]{LMZ}). A set $B\su \rr$ is {\it
$d$-closed} if $\rr \se B$ is $d$-open; that is, if $B$ contains every
point $x$ such that $\ol d (B,x)>0$. 
\begin{theorem}\label{t2}
Let $A\su \rr$ be Lebesgue measurable, and suppose that $B\su \rr$ is $d$-closed
and is $G_\de$ in the Euclidean topology. If every null subset of $A$ can be
translated into $B$, then $A$ can be translated into $B$.
\end{theorem}
Theorem \ref{t2} implies Theorem \ref{t1}. Indeed, suppose $A,B$ are measurable,
and every null subset of $A$ can be translated into $B$. Let $B_1 =B\cup
\{ x\in \rr \colon \ol d (B,x)>0 \}$. Then $B_1$ is $d$-closed, and $B_1 \se B$
is null by the Lebesgue density theorem. Let $B_2$ be a $G_\de$ set such that
$B_1 \su B_2$ and $B_2 \se B_1$ is null. Then $B_2$ is also $d$-closed, and
$B_2  \se B$ is a null set. Since every null subset of $A$ can be translated
into $B_2$, $A$ can be translated into $B_2$ by Theorem \ref{t2}. If $A+x\su
B_2$, then $(A\se N)+ x\su B$, where $N=(B_2 \se B)-x$ is a null set,
proving Theorem \ref{t1}.

The condition that $A$ can be translated into $B$ can be formulated in terms
of sums of sets as follows. If $B\su \rr$, then we denote $-B=\{ -x\colon
x\in B\}$ and $B^c =\rr \se B$. The set $\{ a+b\colon a\in A, \ b\in B\}$ is
denoted by $A+B$.
\begin{lemma}\label{l1}
For every $A,B\su \rr$, the following are equivalent.
\begin{enumerate}[{\rm (i)}]
\item $A$ cannot be translated into $B$.
\item $A+(-B^c )=\rr$.
\end{enumerate}  
\end{lemma}
\proof 
The set $A$ cannot be translated into $B$ if and only if for every $x\in \rr$
there is an $a\in A$ such that $x+a\notin B$; that is, $x+a\in B^c$. Thus
(i)$\iff x\in (-A)+ B^c$ for every $x\in \rr \iff (-A)+B^c =\rr \iff A+(-B^c )
=\rr$. \hfill $\square$

According to the lemma above, the statement of Theorem \ref{t2} is equivalent
to the following. {\it Suppose that $A$ is measurable, and $B$ is $d$-open and
$F_\si$. If $A+B=\rr$, then there is a null set $N\su A$ such that $N+B=\rr$.}
We will prove somewhat more.
\begin{theorem}\label{t3}
Suppose that $A$ is measurable, $B$ is $F_\si$, and $\ol d (B,x)>0$ for every
$x\in B$. Then there is a null set $N\su A$ such that $N+B=A+B$.
\end{theorem}
We can say more if $A$ and $B$ are compact.
\begin{theorem}\label{t4}
Let $A,B$ be compact subsets of $\rr$ such that $\ol d (B,x)>0$ for every
$x\in B$. Then there is a closed null set $N\su A$ such that $N+B=A+B$.
\end{theorem}
Theorems \ref{t3} and \ref{t4} will be proved in the next section.

We conclude with some remarks and problems.

{\it Remarks on Theorem \ref{t1}.} We may ask if the following, stronger
statement is true.

(T): {\it If $A,B\su \rr$ are measurable, and every null subset of $A$
can be translated into $B$, then there is a countable set $C\su A$ such that
$A\se C$ can be translated into $B$.}

The answer is negative, at least consistently. A set $N$ is called {\it strongly
meager} if every null set can be translated away from $N$. Therefore, if $N$ is
an uncountable strongly meager set, then every null set can be translated into
$\rr \se N$, but $\rr \se C$ cannot be translated into $\rr \se N$ for any
countable set $C$. It is consistent with ZFC that there are strongly
meager sets of cardinality continuum (see \cite[Corollary 8.20, p. 304]{Bu}),
and thus (T) is consistently false.

On the other hand, it is also consistent with ZFC that every strongly
meager set is countable (see \cite[Metatheorem 9.2, p. 385]{Bu}). Thus (T)
is consistently true in the case $A=\rr$. This motivates the following
questions.

\begin{problem}
Is {\rm (T)} consistent with {\rm ZFC}?
\end{problem}

\begin{problem}
Suppose that $A,B\su \rr$ are Borel. Is it true that if every null subset
of $A$ can be translated into $B$, then there is a countable set $C\su A$
such that $A\se C$ can be translated into $B$?
\end{problem}

{\it Remarks on Theorem \ref{t3}.} It is clear that some condition on $B$ is
necessary: if $|A|>0$ and $B$ is countably infinite, then the conclusion of
the theorem is false. More generally, suppose that $B$ is nonempty and has the
property that $N+B$ is null for every null set $N$. It is obvious that if
$|A|>0$, then $N+B\ne A+B$ for every  null set $N\su A$. It
is known that under Martin's axiom there is such a set $B$ of cardinality
continuum \cite{FT}. See also \cite[Theorem 12]{EKM}, where such a set is
constructed under (CH). By \cite[Theorem 13]{EKM}, it is also consistent
with ZFC that there is a coanalytic set $B$ with the properties above.

\begin{problem}
Let {\rm (S)} denote the following statement: if $A,B\su \rr$ are measurable
and $B$ is uncountable, then there is a null set $N\su A$ such that
$N+B=A+B$. Is {\rm (S)} consistent with {\rm ZFC}?
\end{problem}

Note that (S) implies (T). Indeed, suppose that $A,B$ are measurable, and
every null subset of $A$ can be translated into $B$. If $B^c$ is countable,
then $A\se B^c$ can be translated into $B$; in fact, $A\se B^c \su B$. Suppose
that $B^c$ is uncountable. If $N$ can be translated into $B$, then
$N+(-B^c )\ne \rr$ by Lemma \ref{l1}. If this is true for every null set
$N\su A$, then $A+(-B^c)\ne \rr$ by (S), and thus $A$ can be translated into
$B$. Therefore, if (S) is consistent with ZFC then so is (T).

{\it Remarks on Theorem \ref{t4}.} 
The Erd\H os-Kunen-Mauldin theorem states that {\it if $B\su \rr$ is a nonempty
perfect set, then there is a perfect null set $N\su \rr$ such that $N+B=\rr$}.
(See \cite[Theorem 1]{EKM} and also \cite{E}, where a stronger statement is
proved.) A common generalization of the Erd\H os-Kunen-Mauldin theorem and
Theorem \ref{t4} would be the following: if $A,B$ are nonempty perfect sets,
then there is a {\it closed} null set $N\su A$ such that $N+B=A+B$.
Unfortunately, this is false, as we show in Section 3. 

\begin{problem}
Suppose that $A,B\su \rr$ are perfect. Is it true that $N+B=A+B$ for
a suitable (not necessarily closed) null set $N\su A$?
\end{problem}

\section{Proof of Theorems \ref{t3} and \ref{t4}}
\begin{lemma}\label{l2}
If $A$ is measurable and $0<|A|<\infty$, then for every $\ep >0$ there
is a compact set $P\su A$ such that $|A\se P|<\ep$ and $\ol d (P,x)\ge 1/2$ for
every $x\in P$.
\end{lemma}
\proof 
Let $\ep >0$ be given, and let $\ep _0 ,\ep _1 ,\ldots$ be positive numbers
such that $\sum_{k=0}^\infty \ep _k <\ep$. Let $K$ be a compact subset of $A$
such that $|A\se K|<\ep _0$.

Let $\ii$ denote  the set of closed intervals $I$ such that $|I|<1$ and
$|I\cap K|>|I|/2$. Then $\ii$ is a Vitali cover of the set of density
points of $K$. By the Vitali covering theorem, there is a sequence
$I_1^1 ,I_2^1 ,\ldots$ of pairwise disjoint elements of $\ii$ such that they
cover a.e. density point of $K$, hence a.e. point of $K$. Then there
is an integer $N_1$ such that $|K\se P_1 |< \ep _1$, where $P_1 =
\bigcup_{n=1}^{N_1} I_n^1$.

Repeating the argument above in each of the intervals $I_n^1$ we obtain the
pairwise disjoint closed intervals $I_1^2 \stb I_{N_2}^2$ and the set 
$P_2 = \bigcup_{n=1}^{N_2} I_n^2$ with the following properties:
\begin{enumerate}[$\bullet$]
\item For every $1\le n\le N_2$ there is an $m$ such that $I_n^2 \su I_m^1$;
\item $|I_n^1 \cap P_2 \cap K|>|I_n^1 |/2$ for every $n=1\stb N_1$;
\item $|I_n^2 |<1/2$ and $|I_n^2 \cap K|> \tfrac{2}{3} \cd |I_n^2|$ for every
$1\le n\le N_2$;
\item $|K \se P_2 | <\ep _2$. 
\end{enumerate}
Repeating this construction we obtain, for every $k\ge 2$, the
pairwise disjoint closed intervals $I_1^k \stb I_{N_k}^k$ and the set
$P_k = \bigcup_{n=1}^{N_k} I_n^k$ with the following properties:
\begin{enumerate}[$\bullet$]
\item For every $1\le n\le N_k$ there is an $m$ such that $I_n^k \su I_m^{k-1}$;
\item $|I_n^j \cap P_k \cap K|>\tfrac{j}{j+1}\cd |I_n^j |$ for every
$j=1\stb k-1$
and $n=1\stb N_j$;
\item $|I_n^k |<1/k$ and $|I_n^k \cap K|> \tfrac{k}{k+1} \cd|I_n^k|$ for every
$1\le n\le N_k$;
\item $|K \se P_k | <\ep _k$. 
\end{enumerate}
Let $P=\bigcap_{k=1}^\infty P_k$. Then $P$ is closed, $P\su K\su A$, $|A\se P|
<\ep$, and $\ol d (P,x)\ge 1/2$ for every $x\in P$. Indeed, if $x\in P$, then
there are integers $n_k$ such that $x\in I_{n_k}^k$ $(k=1,2,\ldots )$. For every
$k$ we have
$$|I_{n_k}^k \cap P|=\lim_{m\to \infty} \left| I_{n_k}^k \cap P_m \right| \ge
\tfrac{k}{k+1} \cd |I_{n_k}^k |.$$
Putting $J_k =[x-h_k ,x+h_k ]$, where $h_k =\min \{ h\ge 0\colon I_{n_k}^k  \su
[x-h,x+h] \}$, we obtain $|J_k |\le 2|I_{n_k}^k |$ and $\liminf_{k\to \infty} |J_k
\cap P|/|J_k|\ge 1/2$, proving $\ol d (P,x)\ge 1/2$. \hfill $\square$

The following lemma is a variant of the Luzin-Menchoff theorem
\cite[Theorem 6.4, p. 20]{Br}).
\begin{lemma}\label{l3}
Suppose that $A$ is measurable, $F\su A$ is compact, and $\ol d (A,x)>0$ for
every $x\in F$. Then there is a compact set $P$ such that $F\su P\su A$ and
$\ol d (P,x)>0$ for every $x\in P$.
\end{lemma}
\proof Let $a= \min F$ and $b=\max F$. We may assume that the right upper
density of $A$ at the point $a$ is positive. Indeed, otherwise the left upper
density of $A$ at $a$ must be positive, and then $|A\cap (-\infty , a)|>0$.
In this case there is a density point $a_1$ of $A$ such that $a_1 <a$, and
then we can replace $F$ by $F\cup \{ a_ 1\}$. Similarly, we may assume that
the left upper density of $A$ at the point $b$ is positive.

Let $(a_n ,b_n )$ $(n=1,2,\ldots )$ be the components of $[a,b] \se F$. For
every $n$, let $(x_k^n )_{k=-\infty} ^\infty$ be a sequence such that $a_n <x_k^n
< x_{k+1}^n <b_n$ for every integer $k$, $\lim_{k\to -\infty} x_k^n =a_n$ and
$\lim_{k\to \infty} x_k^n =b_n$ and
$$\lim_{k\to -\infty} (x^n_{k+1} -x^n_k )/(x_k^n -a_n )=\lim_{k\to \infty} (x^n_{k+1} -x^n_k )/(b_n -x_{k+1}^n )=0.$$
By Lemma \ref{l2}, for every $n$ and $k$ there is a compact set
$P_k^n \su [x_k^n ,x_{k+1}^n ] \cap A$ such that $|P_k^n | \ge |[x_k^n , x_{k+1}^n ]
\cap A|/2$, and $\ol d (P_k^n ,x)\ge 1/2$ for every $x\in P_k^n$. (If
$[x_k^n , x_{k+1}^n ] \cap A$ is null, then we choose $P_k^n =\emp$.)

We put $P=F\cup \bigcup_{n=1}^\infty \bigcup_{k=-\infty}^\infty P_k^n$.
Then $P$ is compact, $F\su P\su A$, and it is easy to check that
$\ol d (P,x)>0$ for every $x\in P$. \hfill $\square$

\subsection*{Proof of Theorem \ref{t4}}
Let $A,B\su \rr$ be compact, and suppose that the upper density of $B$ is
positive at every point of $B$. 

Let $\nn$ denote the family of those closed subsets of $A$ for which
$N+B=A+B$. Then $\nn$ is nonempty, as $A\in \nn$. If $N_1 \sp N_2 \sp
\ldots$ is a nested sequence of elements of $\nn$ and $N=\bigcap_{i=1}^\infty
N_i$, then $N\in \nn$. Indeed, let $x\in A+B$ be arbitrary. For every $i$
we have $N_i \in \nn$, and thus $x=y_i +z_i$, where $y_i \in N_i$ and
$z_i \in B$. There is a sequence of indices $i_k$ such that the subsequences
$y_{i_k}$ and $z_{i_k}$ are convergent. If $y_{i_k} \to y$ and $z_{i_k} \to z$, then
$y\in N$, $z\in B$ and $y+z=x$, proving $x\in N+B$. Since this is true for every
$x\in A+B$, we have $N+B\sp A+B$ and $N\in \nn$.

It is well known that every transfinite and strictly decreasing sequence of 
closed subsets of $\rr$ is countable. (See \cite[Theorem 51, p. 66]{S}.)
Therefore, the intersection of every transfinite and strictly decreasing
sequence of the elements of $\nn$ belongs to $\nn$. This implies that $\nn$
contains a minimal element (with respect to inclusion). Let $N \in \nn$ be
a minimal element of $\nn$; we prove that $N$ is null.

Suppose $|N|>0$, and let $a\in N$ be a density point of $N$. Since $N$
is minimal, $N_n =N\se (a-1/n,a+1/n)\notin \nn$ for every $n=1,2,\ldots$.
Thus $N_n +B\ne A+B=N+B$, and we can choose an element $c_n \in (N+B)\se
(N_n +B)$. Then $c_n =a_n +b_n$ for suitable $a_n \in N$ and $b_n \in B$,
and whenever $c_n =x+y$ with $x\in N$ and $y\in B$, then $x\in (a-1/n,a+1/n)$.
In particular, we have $a_n \in (a-1/n,a+1/n)$.

Turning to a suitable subsequence, we may assume that the sequence $b_n$
is convergent. Suppose $b_n \to b$, then $b\in B$. We have $\ol d (B,b)>0$ by
assumption, and thus the set $b-B$ has positive upper density at the origin.
Since the origin is a density point of $N-a$, it follows that
$(N-a)\cap (b-B)$ is of positive measure. Let $|(N-a)\cap (b-B)|=m$.

It is well known, and is easy to prove, that if $E,F\su \rr$ are measurable,
then the function $h\mapsto |E\cap (F+h)|$ is continuous. Thus the function
$h\mapsto m(h)= |(N-a)\cap (b+h-B)|$ is continuous. Since $m(0)=m>0$,
there is a $\de >0$ such that $|(N-a)\cap (b+h-B)|>m/2$ for every $|h|<\de$.

Put $Q_h =(N-a)\cap (b+h-B)$. If $u\in Q_h$, then $u+a\in N$, $b+h-u\in B$,
and $a+b+h=(u+a)+(b+h-u)$. If $|h|<\de$, then it follows from
$|Q_h |>m/2$ that there are elements $u_1 ,u_2 \in Q_h$ such that
$|u_1 -u_2 |>m/4$, since otherwise $Q_h$ could be covered by a segment of
length $m/4$, which is impossible.

Since $c_n \to a+b$, there is an $n$ such that $|c_n -a-b|<\de$ and $n>8/m$. 
Fix such an index $n$, and put $h=c_n -a-b$. Then there are elements
$u_1 ,u_2 \in Q_h$ such that $|u_1 -u_2 |>m/4$, and
$$c_n =a+b+h=(u_1 +a)+(b+h-u_1 )=(u_2 +a)+(b+h-u_2) .$$
Here $u_i +a \in N$ and $b+h-u_i \in B$ $(i=1,2)$, and thus  $u_1 +a, u_2 +a
\in (a-1/n,a+1/n)$. This, however, contradicts $2/n <m/4$. This contradiction
proves that $|N|=0$. \hfill $\square$

\subsection*{Proof of Theorem \ref{t3}}
Suppose that $A$ is measurable, $B$ is $F_\si$, and the upper density of $B$
is positive at every point of $B$.

We have $A=N_0 \cup A'$, where $N_0$ is null and $A'$ is $F_\si$. Let
$A'=\bigcup_{k=1}^\infty F_k$ and $B=\bigcup_{k=1}^\infty P_k$, where $F_k$ and
$P_k$ are compact. We may assume that $F_1 \su F_2 \su \ldots$ and
$P_1 \su P_2 \su \ldots$. It is enough to prove that for every $k=1,2, \ldots$
there is a null set $N_k \su A$ such that $N_k +B\sp F_k +P_k$. Indeed, if
this is true, then $N=\bigcup_{k=0}^\infty N_k$ is null, and
\begin{align*}
N+B= &(N_0 +B) \cup \bigcup_{k=0}^\infty (N_k +B)\sp (N_0 +B) \cup
\bigcup_{k=0}^\infty (F_k +P_k ) =\\
&(N_0 +B) \cup (A' +B)=A+B.
\end{align*}
Fix $k$. We have $\ol d (B,x)>0$ for every $x\in P_k$. By Lemma
\ref{l3}, there is a compact set $Q_k$ such that $P_k \su Q_k\su B$ and
$\ol d (Q_k ,x)>0$ for every $x\in Q_k$. By Theorem \ref{t4}, there is a closed
null set $N_k \su F_k$ such that $N_k +Q_k = F_k +Q_k $. Then $N_k +B \sp
N_k +Q_k = F_k +Q_k \sp F_k +P_k$. \hfill $\square$

\section{An example}
\begin{theorem}\label{ex1}
Let $A$ be a nowhere dense compact perfect set of positive measure. Then there
exists a compact perfect set $B$ such that $N+B\ne A+B$ for every closed null
subset $N$ of $A$.
\end{theorem}
\proof
We denote by $\cl E$ and $\inter E$ the closure and the interior of the set
$E$. For every positive integer $k$ let $i_k$ denote the largest integer $i$
such that $2^i \mid k$. Then $i_k <k$ for every $k$, and for every nonnegative
integer $m$ there are infinitely many $k$ with $i_k =m$.

Let $X$ be a countable dense subset of $A$, and let $x_0 ,x_1 ,\ldots$ be an
enumeration of $X$. We construct a sequence $(y_k )$ such that $B=\cl (\{ y_k \} )$ is perfect and compact, and whenever $x_k +y_k =x+y$, where $x\in A$ and $y
\in B$, then $x=x_k$ and $y=y_k$. It will follow that if $N\su A$ and
$N+B=A+B$, then $x_k \in N$ for every $k$; that is, $X\su N$. Thus $N$ must
be dense in $A$ and, consequently, if $N$ is closed that it must be equal to
$A$, and cannot be a null set. 

We turn to the construction of the numbers $y_k$.
Put $y_0 =0$. Since $(x_0 +y_0 -A)^c$ is an everywhere dense
open set, we can choose pairwise disjoint closed intervals $I_n^0 \su
(x_0 +y_0 -A)^c$ such that
$$\cl \left( \bigcup_{n=1}^\infty I_n^0 \right) =\{ y_0 \} \cup
\bigcup_{n=1}^\infty I_n^0 .$$
Suppose that $k>0$, and we have defined the points $y_0 \stb y_{k-1}$ and the
closed intervals $I_1^{k-1} , I_2^{k-1} , \ldots$ such that
$$\cl \left( \bigcup_{n=1}^\infty I_n^{k-1} \right) =\{ y_0 \stb y_{k-1} \} \cup
\bigcup_{n=1}^\infty I_n^{k-1} .$$
The set $G_k =\bigcap_{i=0}^{k-1} (y_i -x_k +A)^c$ is everywhere dense and open.
Thus we can choose a point $y_k \in G_k \cap \bigcup_{n=1}^\infty \inter I_n^{k-1}$
such that $0<|y_k -y_{i_k}|<1/2^k$.

Since $(x_k +y_k -A)^c$ is an everywhere dense open set, we can choose pairwise
disjoint closed intervals
$$I_n^k \su (x_k +y_k -A)^c \cap \bigcup_{n=1}^\infty \inter I_n^{k-1}$$
such that
\begin{equation}\label{e4}
\cl \left (\bigcup_{n=1}^\infty I_n^k \right) =\left (\bigcup_{n=1}^\infty I_n^k
\right) \cup \{ y_0 \stb y_k \} \su (x_k +y_k -A)^c \cup \{ y_0 \stb y_k \} .
\end{equation}
In this way we defined the points $y_k$ and the intervals $I_n^k$ for every
$k\ge 0$ and $n\ge 1$. Put $Y=\{ y_0 ,y_1 ,\ldots \}$ and $B=\cl (Y )$. Then
$B$ is a compact set, as $Y$ is bounded. We show that $B$ is perfect. Indeed,
if $y\in B\se Y$, then $y$ is not an isolated point of $B$, since $y$ is the
limit of a subsequence of $(y_k)$. On the other hand, if $y=y_m$, then
$i_k =m$ for infinitely many $k$. Thus $0<|y_k -y_m|<1/2^k$ for infinitely
many $k$ proving that $y_m$ is not an isolated point of $B$ either.

It follows from the construction that $B\su \cl (\bigcup_{n=1}^\infty I_n^k )$
for every $k$. Therefore, by \eqref{e4} we obtain
\begin{equation}\label{e5}
B\su (x_k +y_k -A)^c \cup \{ y_0 \stb y_k \} \qquad (k=0,1,\ldots ).
\end{equation}
Suppose that $x_k +y_k =x+y$, where $x \in A$ and $y\in B$. Then $y =
x_k +y_k -x \in x_k +y_k -A$, and thus \eqref{e5} gives $y \in  
\{ y_0 \stb y_k \}$. If $y \ne y_k$, then $y_k =x+y -x_k \in
\bigcup_{i=0}^{k-1} (y_i -x_k +A) =(G_k )^c$, which contradicts $y_k \in G_k$.
Thus $y =y_k$ and, consequently, $x =x_k$. As we saw above, this implies
that whenever $N+B=A+B$, where $N\su A$, then $N$ must be dense in $A$, and
thus it cannot be a closed null subset of $A$. \hfill $\square$

\subsection*{Acknowledgments}
The author was supported by the Hungarian National Foundation for Scientific
Research, Grant No. K146922.

\begin{small}\noindent
M. Laczkovich\\
{\sc Professor emeritus of Mathematics at the\\
Department of Analysis, E\"otv\"os Lor\'and University, Budapest and\\
Department of Mathematics, University College London\\
E-mail: {\tt miklos.laczkovich@gmail.com}}
\end{small}

\end{document}